\documentclass{amsart}
\usepackage{hyperref}
\usepackage{amsthm,amsmath,amssymb}
\usepackage{fancyhdr}
\usepackage[outer=1.5in, inner=1.25in, top=1.25in, bottom=1.25in]{geometry}
\usepackage[shortlabels]{enumitem}
\numberwithin{equation}{section}
\usepackage[english]{babel}
\newcommand{\Aut}{\textup{Aut}}
\newcommand{\fpr}{\textup{fpr}}


\begin{document}
	
	\newtheorem{teo}{Theorem}[section]
	\newtheorem{de}[teo]{Definition}
	\newtheorem{cor}[teo]{Corollary}
	\newtheorem{lem}[teo]{Lemma}

\title[Edge fixity]{On the number of fixed edges of automorphisms of vertex-transitive graphs of small valency}

\author[Barbieri]{Marco Barbieri}
\address{Dipartimento di Matematica "Felice Casorati", University of Pavia, Via Ferrata 5, 27100 Pavia, Italy} 
\email{marco.barbieri07@universitadipavia.it}
\author[Grazian]{Valentina Grazian}
\address{Dipartimento di Matematica e Applicazioni, University of Milano-Bicocca, Via Cozzi 55, 20125 Milano, Italy} 
\email{valentina.grazian@unimib.it}
\author[Spiga]{Pablo Spiga}
\address{Dipartimento di Matematica e Applicazioni, University of Milano-Bicocca, Via Cozzi 55, 20125 Milano, Italy} 
\email{pablo.spiga@unimib.it}
	
\begin{abstract}
	We prove that, if $\Gamma$ is a finite connected $3$-valent vertex-transitive, or $4$-valent vertex- and edge-transitive graph, then either $\Gamma$ is part of a well-understood family of graphs, or every non-identity automorphism of $\Gamma$ fixes at most $1/3$ of the edges. This answers a question proposed by Primo\v{z} Poto\v{c}nik and the third author.
\end{abstract}	
	
\maketitle	
\section{Introduction}
Poto\v{c}nik and Spiga have proved in \cite{PS.fixedVertex} that, if $\Gamma$ is a finite connected $3$-valent vertex-transitive graph, or a $4$-valent vertex- and edge-transitive graph then, unless $\Gamma$ belongs to a well-known family of graphs, every non-identity automorphism of $\Gamma$ fixes at most $1/3$ of the vertices. In the same work, they have proposed  a  similar investigation with respect to the edges of the graph, see~\cite[Problem 1.7]{PS.fixedVertex}. In this paper we solve this problem.

\begin{teo}\label{thrm:main1}
	Let $\Gamma$ be a finite connected $4$-valent vertex- and edge-transitive  graph admitting a non-identity automorphism fixing more than $1/3$ of the edges. Then  one of the following holds:
	\begin{enumerate}
		\item\label{thrm:main1eq} $\Gamma$ is isomorphic to the complete graph on $5$ vertices;
		\item\label{thrm:main2eq} $\Gamma$ is isomorphic to a Praeger-Xu graph $C(r,s)$, for some $r$ and $s$ with $3s < 2r-3$.
	\end{enumerate}
\end{teo}

\begin{teo}\label{thrm:main2}
	Let $\Gamma$ be a finite connected $3$-valent vertex-transitive  graph admitting a non-identity automorphism fixing more than $1/3$ of the edges. Then $\Gamma$ is isomorphic to a Split Praeger-Xu graph $SC(r,s)$,  for some $r$ and $s$ with $3s < 2r-3$.
\end{teo}
We refer to Section~\ref{secPrXu} for the definition of the ubiquitous Praeger-Xu graphs and for their splittings. The bound in Theorem~\ref{thrm:main2} is sharp. For instance, each $3$-valent graph admitting a non-identity automorphism fixing setwise a complete matching  has the aforementioned property.  For valency $4$,  we conjecture that the bound $1/3$ in Theorem~\ref{thrm:main1} can be strengthen to $1/4$, by eventually including some more small exceptional graphs in part~\eqref{thrm:main1eq}.

Theorem~\ref{thrm:main1} and~\ref{thrm:main2}  rely on the following group-theoretic fact:
\begin{teo}[\cite{PS.transPerm}, Theorem~1.1] \label{O2G=1}Let $G$ be a finite transitive permutation group on $\Omega$ containing no non-identity normal subgroup of order a power of $2$. Suppose there exists $\omega \in \Omega$ such that the stabilizer $G_\omega$ of $\omega$ in $G$ is a $2$-group. Then, every non-identity element of $G$ fixes at most $1/3$ of the points.
\end{teo}

The main results of this paper and the results in~\cite{PS.fixedVertex} show that, besides small exceptions or well-understood families of graphs, non-identity automorphisms of 3-valent or 4-valent vertex-transitive graphs cannot fix many vertices or edges. Where ``too many'' in this context has to be considered as a linear function on the number of vertices (and, even then, with a small caveat for $4$-valent graphs, because of the assumption of edge-transitivity).   In our opinion, the difficulty in having a unifying theory of vertex-transitive graphs of small valency admitting non-identity automorphisms fixing too many vertices or edges is due to our lack of understanding possible generalizations of Praeger-Xu graphs. That is, vertex-transitive graphs of bounded valency playing the role of Praeger-Xu graphs. It seems to us that this is a recurrent problem in the theory of groups acting on finite graphs of bounded valency. A general investigation in this direction, but with much weaker bounds and only for arc-transitive graphs, is in~\cite{MR4343059}.

Investigations on the number of fixed points of graph automorphisms do have interesting applications. For instance, very recently Poto\v{c}nik, Toledo and Verret~\cite{PTV} pivoting on the results in~\cite{PS.fixedVertex} have proved remarkable results on the cycle structure of general automorphisms of $3$-valent vertex-transitive and $4$-valent arc-transitive graphs.

\subsection{Structure of the paper}\label{structure}
In Section~\ref{sec2new}, we introduce some basic terminology and, in particular, we introduce the Praeger-Xu graphs and their splittings. Then, we start in Section~\ref{sec2} with some preliminary results.  In Section~\ref{tetrageneral}, we prove Theorem~\ref{thrm:main1} and, in Section~\ref{secBuddy}, we prove Theorem~\ref{thrm:main2}.

\section{The players}\label{sec2new}

\subsection{Basic group-theoretic notions}\label{notation:group}
 Given  a permutation $g$ on a set $\Omega$, we write $\mathrm{Fix}(\Omega,g)$ for the set of \textbf{fixed points of $g$}, i.e.
\begin{equation*}
	\mathrm{Fix}(\Omega,g) = \left\lbrace \omega\in \Omega \;|\;\; \omega ^g = \omega \right\rbrace,
\end{equation*}
and we write $\mathrm{fpr} (\Omega,g)$ for the \textbf{fixed-point-ratio of $g$}, i.e.
\begin{equation*}
	\mathrm{fpr}(\Omega, g) = \frac{|\mathrm{Fix}(\Omega,g)|}{|\Omega|}.
\end{equation*}

A permutation group $G$ on $\Omega$ is said to be \textbf{semiregular} if the identity is the only element fixing some point. When $G$ is semiregular and transitive on $\Omega$, the group $G$ is \textbf{regular} on $\Omega$.

Given a permutation group $G$ of $\Omega$ and a partition $\Sigma$ of $\Omega$, we say that $\Sigma$ is $G$-\textbf{invariant} if $\sigma^g\in \pi$, for every $\sigma\in \Sigma$. Given a normal subgroup $N$ of $G$, the orbits of $N$ on $\Omega$ form a $G$-invariant partition, which we denote by $\Omega/ N$.

We present here a useful lemma involving the notion just defined.
\begin{lem}[\cite{PS.fixedVertex}, Lemma~1.16]\label{fprQuotient}
	 Let $G$ be a group acting transitively on $\Omega$ and let $\Sigma$ be a $G$-invariant partition of $\Omega$. For $g\in G$, let $g^\Sigma$ be the permutation of $\Sigma$ induced by $g$. Then $
	\fpr(\Omega, g) \leq \fpr(\Sigma, g^\Sigma).$
	In particular, if $N\unlhd G$, then 
	$\fpr(\Omega, g) \leq \fpr(\Omega/N, Ng)$.
\end{lem}

\subsection{Basic graph-theoretic notions}\label{notation:gr}
In this paper, a \textbf{digraph} is binary relation
\begin{equation*}
	\Gamma=(V\Gamma,A\Gamma),
\end{equation*}
where $A\Gamma\subseteq V\Gamma\times V\Gamma$. We refer to the elements of $V\Gamma$ as  \textbf{vertices} and to the elements of $A\Gamma$ as arcs. A \textbf{graph} is a finite simple undirected graph, i.e. a pair
\begin{equation*}
	\Gamma=(V\Gamma,E\Gamma),
\end{equation*}
where $V\Gamma$ is a finite set of \textbf{vertices}, and $E\Gamma$ is a set of unordered pairs of $V\Gamma$, called \textbf{edges}. In particular, a graph can be thought of as a digraph where the binary relation is symmetric and contains no loops. Given a non-negative integer $s$, an \textbf{$s$-arc} of $\Gamma$ is an ordered set of $s+1$ adjacent vertices with any three consecutive elements pairwise distinct. When $s=0$, an $s$-arc is simply a vertex of $\Gamma$; when $s=1$,  an $s$-arc is simply an \textbf{arc}, that is, an oriented edge. 

The \textbf{girth} of $\Gamma$, denoted by $g(\Gamma)$, is the minimum length of a cycle in $\Gamma$.

We denote by $\Gamma(v)$ the \textbf{neighbourhood} of the vertex $v$. The size of $|\Gamma(v)|$ is  the \textbf{valency} of $v$. We are mainly dealing with \textbf{regular} graphs, that is, with graphs where $|\Gamma(v)|$ is constant as $v$ runs through the elements of $V\Gamma$. In these cases, we refer to the valency of the graph.

 Let $\Gamma$ be a graph, let $G$ be a subgroup of the automorphism group $\Aut(\Gamma)$ of $\Gamma$, let $v \in V\Gamma$ and let $w\in \Gamma(v)$. We denote by $G_v$ the \textbf{stabilizer} of the vertex $v$, by $G_{\{v,w\}}$ the setwise stabilizer of the edge $\{v,w\}$, by $G_{vw}$ the pointwise stabilizer of the edge $\{v,w\}$ (that is, the stabilizer of the arc $(v,w)$ underlying the edge $\{v,w\}$). The group $G_v$ acts on $\Gamma(v)$ and we denote by $G_v^{[1]}$ the kernel of the action of $G_v$ on $\Gamma(v)$. Now, the permutation group induced by $G_v$ on $\Gamma(v)$ is denoted by $G_v^{\Gamma(v)}$ and we have
$$G_v^{\Gamma(v)}\cong\frac{G_v}{G_v^{[1]}}.$$
 
When $G$ acts transitively on the set of $s$-arcs of $\Gamma$, we say that $G$ is $s$-\textbf{arc-transitive}. When $s=0$, we say that $G$ is \textbf{vertex-transitive} and, when $s=1$, we say that $G$ is \textbf{arc-transitive}.  Moreover, when $G$ acts regularly on the set of $s$-arcs of $\Gamma$ we emphasis this fact by saying that $G$ is $s$-\textbf{arc-regular}.

When $G$ acts transitively on $E\Gamma$, we say that $G$ is \textbf{edge-transitive}. Finally, when $G$ is edge- and vertex-transitive, but not arc-transitive, we say that $G$ is \textbf{half-arc-transitive}. This name comes from the fact that $G$ has two orbits on ordered pairs of adjacent vertices of $\Gamma$ (a.k.a. arcs), each orbit containing precisely one of the two arcs underlying  each edge.

We say that $\Gamma$ is vertex-, edge- or arc-transitive when $\mathrm{Aut}(\Gamma)$ is vertex-, edge- or arc-transitive.

Let $G$ be a finite group and let $S$ be a subset of $G$. The \textbf{Cayley digraph} on $G$ with connection set $S$ is the digraph $\Gamma:=\mathrm{Cay}(G,S)$ having vertex set $G$ and where $(g,h)\in A\Gamma$ if and only if $gh^{-1}\in G$. Now, $\mathrm{Cay}(G,S)$ is a symmetric binary relation  if and only if $S$ is inverse closed, that is, $S=S^{-1}$ where $S^{-1}:=\{s^{-1}\mid s \in S\}$. Observe that the right regular representation of $G$ acts as a group of automorphisms on $\mathrm{Cay}(G,S)$.

\subsection{Praeger-Xu graphs}\label{secPrXu}
In this and in the next section, we introduce the infinite families of graphs appearing in our main theorems. We introduce the $4$-valent \textbf{Praeger-Xu graphs} $C(r,s)$ through their directed counterpart defined in \cite{Prager.arcTransDi}. Further details on Praeger-Xu graphs can be found in~\cite{GardinerPraeger.charTetraSymGraphs, JajcayPW.cycleStrPrXuGraphs, PragerXu.twicePrimeValency}. We also advertise~\cite{MR4321645}, where the authors have begun a thorough investigation of Praeger-Xu graphs, motivated by the recurrent appearance of these objects in the theory of groups acting on graphs.

\par Let $r\geq 3 $ be an integer. Then $\vec{C}(r,1)$ is the lexicographic product of a directed cycle of length $r$ with the edgeless graph on $2$ vertices. In other words, $V\vec{C}(r,1)=\mathbb{Z}_r\times \mathbb{Z}_2$, and the two arcs starting in $(x,i)$ end in $(x+1,0)$ and in $(x+1,1)$. For any $2\leq s\leq r-1$, $V\vec{C}(r,s)$ is defined as the set of all $(s-1)$-arcs of $\vec{C}(r,1)$, and $(v_0,v_1,\dots,v_{s-1})\in V\vec{C}(r,s)$ is the beginning point of the two arcs ending in $(v_1, v_2,\dots,v_{s-1},u)$ and in $(v_1, v_2,\dots,v_{s-1},u')$, where $u$ and $u'$ are the two vertices of $\vec{C}(r,1)$ that prolong the $(s-1)$-arc $(v_1, v_2,\dots,v_{s-1})$. The Praeger-Xu graph $C(r,s)$ is then defined as the non-oriented underlying graph of $\vec{C}(r,s)$. It can be verified that $C(r,s)$ is a connected $4$-valent graph with $r2^s$ vertices and $r2^{s+1}$ edges.
\par We  describe the automorphisms of  $C(r,s)$. Some automorphism of $C(r,s)$ arises from the action of $\mathrm{Aut}(C(r,1))$ on the set of $s$-arcs of $C(r,1)$. 
Let $i\in \mathbb{Z}_r$ and let $\tau_i$ be the transposition on $V\vec{C}(r,1)$ swapping the vertices $(i,0)$ and $(i,1)$ and fixing the remaining vertices. Since $\tau_i$ is an automorphism of $\vec{C}(r,1)$, it is immediate to extended the action of $\tau_i$ to $C(r,1)$ and to $C(r,s)$. We define the group
\begin{equation*}
	K = \langle \tau_i \;|\;\; i\in \mathbb{Z}_r \rangle \cong C_2^r, 
\end{equation*}
and throughout this paper the symbol $K$ will always refer to this group for some $C(r,s)$. Focusing on the cyclic nature of the Praeger-Xu graphs, it is also natural to define on $V\vec{C}(r,1)$ the permutations $\rho$ and $\sigma$ as follows
\begin{equation*}
	(x,i)^\rho = (x+1,i), \hspace{4.5mm}\textup{and}\hspace{4.5mm} (x,i)^\sigma = (-x,i).
\end{equation*}
While $\rho$ is an automorphism of $\vec{C}(r,1)$, $\sigma$ is an automorphism of $C(r,1)$ but not of $V\vec{C}(r,1)$. Moreover, observe that the group $\langle \rho,\sigma \rangle$ normalizes $K$. Define
\begin{equation*}
	H = K\langle \rho,\sigma \rangle, \hspace{4.5mm}\textup{and}\hspace{4.5mm} H^+ = K\langle \rho \rangle,
\end{equation*}
and, as for $K$, the symbols $H$ and $H^+$ will always refer to these groups. Clearly $H\cong C_2 \wr D_r$ is a group of automorphisms of $C(r,s)$ and $H^+\cong C_2 \wr C_r$ is a group of automorphisms of $\vec{C}(r,s)$. Moreover, $H$ acts vertex- and edge-transitively on $C(r,s)$ (and so does $H^+$ on $\vec{C}(r,s)$), but not $2$-arc-transitively. 

\begin{lem}\label{AutPrXu}
	Using the notation above,
$	\Aut( \vec{C}(r,s)) = H^+$
	and, if $r\neq 4$, $\Aut(C(r,s))=H$. Moreover,
	\begin{equation*}
	|\Aut(C(4,1)):H|=9,\;\; |\Aut(C(4,2)):H|=3 \;\; \textit{and} \;\; |\Aut(C(4,3)):H|=2.
	\end{equation*}
\end{lem}
\begin{proof}
	It follows from \cite[Theorem 2.8]{Prager.arcTransDi} and \cite[Theorem 2.13]{PragerXu.twicePrimeValency} when $p=2$.
\end{proof}

The Praeger-Xu graphs also admit the following algebraic characterization.
\begin{lem}\label{charPrXu}
	Let $\Gamma$ be a finite connected $4$-valent graph and let $G$ be a vertex- and edge-transitive group of automorphisms of $\Gamma$. If $G$ has an abelian normal subgroup which is not semiregular on $V\Gamma$, then $\Gamma$ is isomorphic to a Praeger-Xu graph $C(r,s)$, for some integers $r$ and $s$.
\end{lem}
\begin{proof}
	It follows by \cite[Theorem 2.9]{Prager.arcTransDi} and \cite[Theorem 1]{PragerXu.twicePrimeValency} upon setting $p=2$.
\end{proof}

\subsection{Split Praeger-Xu graphs}\label{secSplitPrXu}
For our purposes,  the \textbf{Split Praeger-Xu graphs} are obtained form the Praeger-Xu graphs via the splitting operation which was introduced in \cite[Construction 9]{PSV.cubicCensus}, and which we will comment upon in Section \ref{secBuddy}.

 Here we  give an explicit description of $SC(r,s)$. Split any vertex of $\vec{C}(r,s)$ into two copies, say $v_+$ and $v_-$. For any arc of $\vec{C}(r,s)$ of the form $(v,u)$, let $v_+$ be adjacent to $v_-$ and $u_-$. From the complementary perspective, the neighbourhood of $v_-$ is made up of $v_+$ plus the two vertices $w_+$ such that $(w,v)$ is an arc of $\vec{C}(r,s)$.

\section{Preliminary results}\label{sec2}

\subsection{Graph-theoretical considerations}
In this section, we develop our tool box that extends outside the scope of proving our main theorems.

\begin{lem}\label{arcTransGirth}
	Let $\Gamma$ be a connected $k$-valent graph, with $k\geq 3$, and let $G$ be an $s$-arc-transitive group of automorphisms of $\Gamma$. Then the girth of $\Gamma$ is greater than $s$, i.e. $g(\Gamma)\geq s+1$.
\end{lem}
\begin{proof}
	Let $ v_0, v_1, \dots, v_{\ell-2}, v_{\ell-1}$ be the vertices of a cycle of $\Gamma$ of length $\ell$, where $v_i$ is adjacent to $v_{i+1}$, for every $i$. (Computations are performed $\mod \ell$.) We argue by contradiction and we suppose $\ell\le s$. Consider two  $s$-arcs in $\Gamma$ of the form
	\begin{align*}
&	\left( v_0, v_1, \dots, v_{\ell-2}, v_{\ell-1}, v_0, \dots \right),\\
&	\left( v_0, v_1, \dots, v_{\ell-2}, v_{\ell-1}, w, \dots \right),
	\end{align*}
	where $w \in \Gamma(v_{\ell-1})\setminus\{v_0,v_{\ell-2}\}$. As $G$ is transitive on $s$-arcs and the previous $s$-arcs are not $G$-conjugate, we find a contradiction. 
\end{proof}

\begin{lem}\label{transB1}
	Let $\Gamma$ be a finite connected graph and let $v\in V\Gamma$ be a vertex. For each $w\in \Gamma(v)$, let $t_w$ be an automorphism of $\Gamma$ with $v^{t_w}=w$. Then $T:=\langle t_w\mid w\in \Gamma(v)\rangle$ is vertex-transitive on $\Gamma$.
\end{lem}
\begin{proof}Let $u\in V\Gamma$. As $\Gamma$ is connected, we prove the existence of $t_u\in T$ with $v^{t_u}=u$ arguing by induction on the minimal distance $d:=d(v,u)$ from $v$ to $u$ in $\Gamma$. When $d=0$, that is, $v=u$, we may take $t_u$ to be the identity of $T$. Suppose then $d>0$. Let $v_0,\ldots,v_d$ be a path of distance $d$ from $v=v_0$ to $u=v_d$ in $\Gamma.$ Now, $d(v,v_{d-1})=d-1$ and hence, by induction, there exists $t\in T$ with $v^{t}=v_{d-1}$. Set $u':=u^{t^{-1}}$. As $u=v_d\in \Gamma(v_{d-1})$, we have $$u'=u^{t^{-1}}\in \Gamma(v_{d-1})^{t^{-1}}=\Gamma(v_{d-1}^{t^{-1}})=\Gamma(v).$$ 
By hypothesis, $t_{u'}\in T$ and $v^{t'}=u'$. Therefore, $v^{t_{u'}t}=u'^{t}=u$ and we may take $t_u:=t_{u'}t$.
\end{proof}

\begin{lem}[\cite{GoRo}, Lemma~3.3.3]\label{2edgeConnected}
	Let $\Gamma$ be a finite connected vertex-transitive graph of valency $k$. Then $\Gamma$ is $k$-edge-connected, i.e. $\Gamma$ remains connected upon eliminating any $m$ edges, with $m\leq k-1$.
\end{lem}

A general result on the fixed-point-ratio of Cayley graphs can be proven regardless of the valency.
\begin{lem}\label{Cayley}Let $G$ be a finite group, let $S$ be an inverse closed non-empty subset of $G$, let $\Gamma:=\mathrm{Cay}(G,S)$ and let $g\in G\setminus\{1\}$. If $\mathrm{fpr}(E\Gamma,g)\ne 0$, then $g^2=1$ and
	\begin{equation*}
	\mathrm{fpr}(E\Gamma, g)= \frac{|g^G\cap S|}{|S||g^G|}, 
	\end{equation*}
where $g^G:=\{hgh^{-1}\mid h\in G\}$ is the conjugacy class of $g$ in $G$. In particular, $\mathrm{fpr}(E\Gamma,g)\le 1/|S|$ and the	
	equality is attained if and only if $g^G\subseteq S$. 
\end{lem}
\begin{proof}
Suppose $\mathrm{fpr}(E\Gamma,g)\ne 0$. We let ${\bf C}_G(g)$ denote the centralizer of $g$ in $G$.

For each $s\in S$, let $E_s:=\{\{x,sx\}\mid x\in G\}$. Observe that $E_s$ is a complete matching of $\Gamma$ and that $\{E_s\mid s\in S\}$ is a partition of the edge set $E\Gamma$.

Let $s\in S$. Suppose $E_s\cap\mathrm{Fix}(E\Gamma,g)\ne \emptyset$ and fix $\{\bar{x},s\bar{x}\}\in E_s\cap\mathrm{Fix}(E\Gamma,g)$. As $g$ fixes the edge $\{\bar{x},s\bar{x}\}$, we have $\bar{x}g=s\bar{x}$ and $s\bar{x}g=\bar{x}$. We deduce $g^2=1$ and $s=\bar{x}g\bar{x}^{-1}$. In other words, $g$ has order $2$ and $g$ has a conjugate in $S$. Now, for every $\{x,sx\}\in E_s$, with a similar computation, we obtain that $\{x,sx\}\in \mathrm{Fix}(E\Gamma,g)$ if and only if $s=xgx^{-1}$. Thus $\bar{x}g\bar{x}^{-1}=xgx^{-1}$ and $x\in \bar{x}{\bf C}_G(g)$. In particular,
$E_s\cap \mathrm{Fix}(E\Gamma,g)=\{\{\bar{x}h,s\bar{x}h\}\mid h\in {\bf C}_G(x)\}$
and hence
$$|E_s\cap\mathrm{Fix}(E\Gamma,g)|=\frac{|{\bf C}_G(g)|}{2}.$$

The previous paragraph has established that $g$ has order $2$. Moreover, for each $s\in S$, $E_s\cap\mathrm{Fix}(E\Gamma,g)\ne \emptyset$ if and only if $s\in g^G$. Furthermore, in the case that $s\in g^G$, the cardinality of $E_s\cap\mathrm{Fix}(E\Gamma,g)$ does not depend on $s$ and equals $|{\bf C}_G(g)|/2$. Therefore,
\begin{align*}
\mathrm{fpr}(E\Gamma,g)&=
\frac{|g^G\cap S||{\bf C}_G(g)|/2}{|E\Gamma|}=
\frac{|g^G\cap S||{\bf C}_G(g)|/2}{|S||G|/2}=
\frac{|g^G\cap S|}{|S||G:{\bf C}_G(g)|}=
\frac{|g^G\cap S|}{|S||g^G|}.
\end{align*}

Since $|g^G\cap S|\le |g^G|$, we have $\mathrm{fpr}(E\Gamma,S)\le 1/|S|$. Moreover, the equality is attained if and only if $g^G\cap S=g^G$, that is, $g^G\subseteq S$.
\end{proof}

The next lemma studies the nature of fixed edges in a Praeger-Xu graph.
\begin{lem}\label{gInK}
	Let $\Gamma=C(r,s)$ be a Praeger-Xu graph and let $g\in \Aut(\Gamma)$  with $g\ne 1$ and with $\fpr(E\Gamma, g)> 1/3$. Then  $3s < 2r-3$ and, either $g\in K$ or $(r,s)=(r,1)$. In particular, $g$ fixes an edge if and only if $g$ fixes both of its ends. (The group $K$ is defined in Section~$\ref{secPrXu}$.)
\end{lem}
\begin{proof}
The lexicographic product $C(4,1)\cong K_{4,4}$ admits automorphisms $x$ fixing $8$ edges and hence $\fpr(E\Gamma,x)=8/16=1/2>1/3$. (The non-identity elements in $\mathrm{Aut}(C(4,1))$ with $\fpr(E\Gamma,x)>1/3$ are not necessarily in $K$, but they fix an edge if and only if they fix both of its ends.) Similarly, it can be verified that, for every $x\in \mathrm{Aut}(C(4,2))$ with $x\ne 1$, we have $\fpr(E\Gamma,x)\le 8/32=1/4$. Furthermore, for every $x\in \mathrm{Aut}(C(4,3))$ with $x\ne 1$, we have $\fpr(E\Gamma, x)=8/64=1/8$. In particular, when $r:=4$, the result follows from these computations.

	\par Suppose $r\neq 4$. By Lemma~\ref{AutPrXu}, $\Aut(\Gamma)=H=K\langle \rho, \sigma \rangle$. In particular, 
	\begin{equation*}
	g=\tau\rho^i\sigma^\varepsilon, \hspace{4.5mm} \textup{for some \;} \tau\in K, i\in \mathbb{Z}_r, \varepsilon \in \mathbb{Z}_2.
	\end{equation*}
	Denote by $\Delta_x$ the set of $(s-1)$-arcs in $\vec{C}(r,1)$ starting at $(x,0)$ or at $(x,1)$. From the definition of the vertex set of $C(r,s)$, we have $\Delta_x\subseteq VC(r,s)$, $|\Delta_x|=2^s$ and 
	$$VC(r,s)=\bigcup_{x\in\mathbb{Z}_r}\Delta_x.$$ Moreover, $\Delta_x$ is a $K$-orbit and the subgraph induced by $\Gamma$ on $\Delta_x\cup \Delta_{x+1}$ is the disjoint union of cycles of length $4$. Observe that, for any $x\in \mathbb{Z}_r$, 
\begin{align}\label{new1}
\Delta_x^\rho&=\Delta_{x+1}, &\Delta_x^\sigma=\Delta_{-x-s+1}.
\end{align}

We start by proving that $g\in K$.

\smallskip 
	
\noindent\textsc{Suppose $\varepsilon=0$. }
Let $\{a,b\}\in\mathrm{Fix}(E\Gamma,g)$. Replacing $a$ with $b$ if necessary, we may suppose that $a\in \Delta_x$ and $b\in \Delta_{x+1}$, for some $x\in \mathbb{Z}_r$. If $a^g=a$ and $b^g=b$, we have $\Delta_x^g=\Delta_{x}$ and $\Delta_{x+1}^g=\Delta_{x+1}$. Now,~\eqref{new1} yields $x+i=x$ and $(x+1)+i=x+1$, that is, $i=0$. Therefore $g\in K$. Similarly, if $a^g=b$ and $b^g=a$, we have $\Delta_x^g=\Delta_{x+1}$ and $\Delta_{x+1}^g=\Delta_{x}$. Now,~\eqref{new1} yields $x+i=x+1$ and $(x+1)+i=x$, that is, $2=0$. However, this implies $r=2$, which is a contradiction because $r\ge 3$.

\smallskip
	
\noindent\textsc{Suppose $\varepsilon=1$. }Since $\langle\rho,\sigma\rangle$ is a dihedral group of order $2r$, replacing $g$ by a suitable conjugate if necessary, we may suppose that either  $r$ is odd and $i=0$, or $r$ is even and $i\in \{0,1\}$.

Assume $i=0$. Let $\{a,b\}\in\mathrm{Fix}(E\Gamma,g)$. As above, replacing $a$ with $b$ if necessary, we may suppose that $a\in \Delta_x$ and $b\in \Delta_{x+1}$, for some $x\in \mathbb{Z}_r$.  If $a^g=a$ and $b^g=b$, we have $\Delta_x^g=\Delta_{x}$ and $\Delta_{x+1}^g=\Delta_{x+1}$. Now,~\eqref{new1} yields $-x-s+1=x$ and $-(x+1)-s+1=x+1$, that is, $2=0$. However, this gives rise to the contradiction $r=2$. Similarly, if $a^g=b$ and $b^g=a$, we have $\Delta_x^g=\Delta_{x+1}$ and $\Delta_{x+1}^g=\Delta_{x}$. Now,~\eqref{new1} yields $-x-s+1=x+1$ and $-(x+1)-s+1=x$, that is, $2x+s=0$. When $r$ is odd, the equation $2x+s=0$ has only one solution in $\mathbb{Z}_r$ and, when $r$ is even, the equation $2x+s=0$ has either zero or  two solutions in $\mathbb{Z}_r$ depending on whether $s$ is odd or even. Recalling that the subgraph induced by $\Gamma$ on $\Delta_x\cup \Delta_{x+1}$ is a disjoint union of cycles of length $4$, we obtain that
\[
|\mathrm{fpr}(E\Gamma,g)|\le
\begin{cases}
\frac{|\Delta_x|}{|E\Gamma|}=\frac{1}{2r}&\textrm{if }r\textrm{ is odd},\\
2\cdot \frac{|\Delta_x|}{|E\Gamma|}=\frac{1}{r}&\textrm{if }r\textrm{ is even}.
\end{cases}
\]
In both cases, we have $\mathrm{fpr}(E\Gamma,g)\le 1/4$, which is a contradiction.

Assume $i=1$. Observe that this implies that $r$ is even. Here the analysis is entirely similar. Let $\{a,b\}\in\mathrm{Fix}(E\Gamma,g)$. As above, replacing $a$ with $b$ if necessary, we may suppose that $a\in \Delta_x$ and $b\in \Delta_{x+1}$, for some $x\in \mathbb{Z}_r$.  If $a^g=a$ and $b^g=b$, we have $\Delta_x^g=\Delta_{x}$ and $\Delta_{x+1}^g=\Delta_{x+1}$. Now,~\eqref{new1} yields $-(x+1)-s+1=x$ and $-(x+2)-s+1=x$, that is, $2=0$. However, this gives rise to the usual contradiction $r=2$. Similarly, if $a^g=b$ and $b^g=a$, we have $\Delta_x^g=\Delta_{x+1}$ and $\Delta_{x+1}^g=\Delta_{x}$. Now,~\eqref{new1} yields $-(x+1)-s+1=x+1$ and $-(x+2)-s+1=x$, that is, $2x+s+1=0$. As $r$ is even, the equation $2x+s+1$ has either zero or  two solutions in $\mathbb{Z}_r$ depending on whether $s$ is even or odd. Recalling that the subgraph induced by $\Gamma$ on $\Delta_x\cup \Delta_{x+1}$ is a disjoint union of cycles of length $4$, we obtain that
\[
|\mathrm{fpr}(E\Gamma,g)|\le
2\cdot \frac{|\Delta_x|}{|E\Gamma|}=\frac{1}{r}.
\]
Thus, we have $\mathrm{fpr}(E\Gamma,g)\le 1/4$, which is a contradiction.

\smallskip

Since $g\in K$, if $g$ fixes the edge $\{a,b\}\in E\Gamma$, then $g$ fixes both end-vertices $a$ and $b$. It remains to show that $3s<2r-3$. Notice that $\tau_i$ moves precisely those $(s-1)$-arcs of $\vec{C}(r,1)$ that pass through one of the vertices $(i,0)$ or $(i,1)$. Therefore, $\tau_i$, as an automorphism of $C(r,s)$, fixes all but $s2^s$ vertices, thus it fixes all but those $(s+1)2^{s+1}$ edges which are incident with such vertices. Since any element in $K$ is obtained as a product of some $\tau_i$, such an element fixes at most as many edges as a single $\tau_i$. Hence 
	\begin{equation*}
	\frac{1}{3} < \fpr(E\Gamma, g) \leq \fpr(E\Gamma, \tau_i) = \frac{\left(r-(s+1)\right)2^{s+1}}{r2^{s+1}} = 	\frac{r-s-1}{r}.\qedhere
	\end{equation*}
\end{proof}

\begin{lem}\label{GconjugateH}
	Let $\Gamma=C(r,s)$ be a Praeger-Xu graph, let $G$ be a vertex- and edge-transitive group of automorphism of $\Gamma$ containing a non-identity element $g$ fixing more that $1/3$ of the edges and with $G$ not $2$-arc-transitive. Then $G$ is $\Aut(\Gamma)$-conjugate to a subgroup of $H$ as defined in Section~$\ref{secPrXu}$.
\end{lem}
\begin{proof}
	By Lemma~\ref{gInK}, $3s < 2r-3$. If $r\ne 4$, then by Lemma \ref{AutPrXu} we have $G\le \mathrm{Aut}(\Gamma)=H$. When $r=4$, then inequality $3s<2r-3$ implies $s=1$. Now, the veracity of this lemma can be verified with a computation in $\mathrm{Aut}(C(4,1))=\mathrm{Aut}(K_{4,4})=S_4\wr S_2$.
\end{proof}

\begin{lem}[\cite{PS.fixedVertex}, Lemma~1.11]\label{quotientCycle} Let $\Gamma$ be a finite connected  $4$-valent graph, let $G$ be a vertex- and edge-transitive group of automorphisms of $\Gamma$, and let $N$ be a minimal normal subgroup of $G$. If $N$ is a $2$-group and $\Gamma/N$ is a cycle of length at least $3$, then $\Gamma$ is isomorphic to a Praeger-Xu graph $C(r,s)$ for some integers $r$ and $s$.
\end{lem}

\section{Proof of Theorem~$\ref{thrm:main1}$} \label{tetrageneral}
In this section we prove Theorem~\ref{thrm:main1}. Our proof is divided into two cases, depending on whether $\Gamma$ admits a group of automorphisms acting $2$-arc-transitively or not. 

\subsection{Proof of Theorem~\ref{thrm:main1} when $\Gamma$ is $2$-arc-transitive}\label{secTetra2AT}
 The following lemma involves four graphs not yet considered in this paper, so it is worth to spend some ink here to describe them.
\begin{itemize}
	\item The complete graph $K_5$ is the only sporadic example arising in Theorem~\ref{thrm:main1}, its automorphism group is $S_5$ and each transposition in $S_5$ fixes $4$ edges out of $10$.
	\item The graph $K_{5,5}-5K_2$ is obtained deleting a complete matching from the complete bipartite graph $K_{5,5}$, its automorphism group is $S_5\times C_2$ and every non-identity automorphism fixes at most $3$ edges out of $10$.
	\item The  hypercube $Q_4$ is the Cayley graph
	\begin{equation*}
		Q_4 := \mathrm{Cay} (\mathbb{Z}_2 ^4, \{(1,0,0,0),(0,1,0,0), (0,0,1,0), (0,0,0,1)\}).
	\end{equation*}
	A non-identity automorphism  of $Q_4$ fixes at most $8$ edges out of $32$.
	\item The graph $BCH$ is the bipartite complement of the Heawood graph. The  vertices of $BCH$ can be identified with the $7$ points and the $7$ lines of the Fano plane. The incidence in the graph is given by the anti-flags in the plane, i.e. the point $p$ is adjacent to the line $L$ if, and only if, $p\notin L$. The automorphism group of $BCH$ is isomorphic to $\mathrm{SL}_3(2).2$. A  non-identity automorphism  of $BHC$ fixes at most $4$ edges out of $28$.
\end{itemize}

\begin{lem}\label{TetraSmallGirth}
	Let $\Gamma$ be a finite connected $4$-valent $2$-arc-transitive graph of girth at most $4$, i.e. $g(\Gamma)\in \{3,4\}$. Then one of the following holds:
	\begin{enumerate}
		\item\label{eq:TetraSmallGirth1} $g(\Gamma)=3$ and $\Gamma$  is isomorphic to the complete graph $K_5$;
		\item\label{eq:TetraSmallGirth2} $g(\Gamma)=4$ and $\Gamma$ is isomorphic to $K_{4,4}\cong C(4,1)$;
		\item\label{eq:TetraSmallGirth3} $g(\Gamma)=4$ and $\Gamma$ is isomorphic to $K_{5,5}-5K_2$, $Q_4$ or $BCH$.
	\end{enumerate}
\end{lem}
\begin{proof}
Let $v$ be a vertex, let $\Gamma(v)=\{w_1,w_2,w_3,w_4\}$ be its neighbourhood and let $G:=\mathrm{Aut}(\Gamma)$.

	First, assume $g(\Gamma)=3$.  Without loss of generality, suppose $w_1$ and $w_2$ are adjacent. Since $G$ is $2$-arc-transitive, $G_v$ is $2$-transitive on $\Gamma(v)$. Hence $w_i$ is adjacent to $w_j$ for any $i\neq j$. Thus $\Gamma\cong K_5$ and part~\eqref{eq:TetraSmallGirth1} holds.

 Now, suppose $g(\Gamma)=4$. We need to recall the classification arising from \cite[Theorem~3.3]{PW.edgeTransGirth4}. If $\Delta$ is a $4$-valent edge-transitive graph, then one of the following holds
	\begin{enumerate}[(i)]
		\item\label{eq:alpha1} each vertex in $\Delta$ is contained in exactly one $4$-cycle,
		\item\label{eq:alpha2} there exist two distinct vertices $v_1,v_2$ with $\Delta(v_1)=\Delta(v_2)$,
		\item\label{eq:alpha3} $\Delta$ isomorphic to $K_{5,5}-5K_2$, $Q_4$ or $BCH$.
	\end{enumerate}
We consider these three possibilities for $\Gamma$ in turn.	Up to a permutation of the indices, there exists $u\in \Gamma(w_1)\cap\Gamma(w_2)$ such that $(v,w_1,u,w_2)$ is a $4$-cycle. Since $G_v^{\Gamma(v)}$ is $2$-transitive, there exists $g\in G_v$ with $(w_1,w_2)^g=(w_3,w_4)$. Therefore, $(v,w_1,u,w_2)^g=(v,w_3,u^g,w_4)$ is a $4$-cycle different from $(v,w_1,u,w_2)$. Thus part~\ref{eq:alpha1} is excluded. If $\Gamma$ satisfies~\ref{eq:alpha2}, then \cite[Lemma~4.3]{PW.edgeTransGirth4} gives that $\Gamma$ is isomorphic to $C(r,1)$ for some integer $r$. From Lemma~\ref{AutPrXu}, $C(r,1)$ is $2$-arc-transitive only when  $r=4$; therefore we obtain part~\eqref{eq:TetraSmallGirth2}. If $\Gamma$ satisfies part~\ref{eq:alpha3}, then we obtain the examples in part~\eqref{eq:TetraSmallGirth3}.
\end{proof}

\begin{de}\label{def:1}{\rm 
Let $\Gamma$ be a finite connected $4$-valent  graph and let $g$ be an automorphism of $\Gamma$. We partition $E\Gamma$ with respect to the action of $g$.
\begin{itemize}
\item We let $A(\Gamma,g)$ be the set of edges which are pointwise fixed by $g$, that is, $\{a,b\}\in A(\Gamma,g)$ if and only if $\{a,b\}\in E\Gamma$, $a^g=a$ and  $b^g=b$;
\item we let $F(\Gamma,g):=\mathrm{Fix}(E\Gamma,g)\setminus A(\Gamma,g)$, that is, $\{a,b\}\in F(\Gamma,g)$ if and only if $\{a,b\}\in E\Gamma$, $a^g=b$ and  $b^g=a$;
\item we let $N(\Gamma,g):=E\Gamma\setminus \mathrm{Fix}(E\Gamma,g)$.
\end{itemize}

We let $\Gamma[g]$ denote the subgraph of  $\Gamma$ induced by $\Gamma$ on the vertices which are incident with  edges in $A(\Gamma,g)$. The  edge-set of $\Gamma[g]$ is $A(\Gamma,g)$ and its vertices are  $1$-, $2$- or $4$-valent. Given $i\in \{1,2,4\}$, we let $V_i(\Gamma,g)$ denote the set of vertices of $\Gamma[g]$ having valency $i$. 
}
\end{de}

\begin{lem}\label{tetraNiceV}
	Let $\Gamma$ be a finite connected $4$-valent graph of girth $g(\Gamma)\geq 5$ and let $g$ be an automorphism of $\Gamma$. Then $2|F(\Gamma,g)|+4|V_1(\Gamma,g)|+3|V_2(\Gamma,g)|+|V_4(\Gamma,g)|\le |V\Gamma|$.
\end{lem}
\begin{proof}
We let 
\begin{align*}
\mathcal{F}&:=\{v\in V\Gamma\mid \{v,u\}\in F(\Gamma,g) \textrm{ for some }u\in V\Gamma\},\\
\mathcal{N}&:=\{v\in V\Gamma\mid \{v,u\}\in N(\Gamma,g)\textrm{ for some }u\in V\Gamma\}.
\end{align*}
Since $V_1(\Gamma,g),V_2(\Gamma,g),V_4(\Gamma,g),\mathcal{F},\mathcal{N}$ are pairwise disjoint and since $|\mathcal{F}|=2|F(\Gamma,g)|$, it suffices to show that
$|\mathcal{N}|\ge 3|V_1(\Gamma,g)|+2|V_2(\Gamma,g)|$.

We construct an auxiliary graph $\Delta$. The vertex set of $\Delta$ is $V_1(\Gamma,g)\cup V_2(\Gamma,g)\cup\mathcal{N}$ and we declare a vertex $v\in V_1(\Gamma,g)\cup V_2(\Gamma,g)$ adjacent to a vertex $u\in\mathcal{N}$ if $\{v,u\}\in E\Gamma$. By construction, $\Delta$ is bipartite with parts $V_1(\Gamma,g)\cup V_2(\Gamma,g)$ and $\mathcal{N}$.

Given $v\in V_1(\Gamma,g)$, the automorphism $g$ acts as a $3$-cycle on $\Gamma(v)$. Let $v_1,v_2,v_3\in \Gamma(v)$ forming the $3$-cycle of $g$. Then $\{v,v_1\},\{v,v_2\},\{v,v_3\}\in N(\Gamma,g)$ and hence $v_1,v_2,v_3\in\mathcal{N}$. This shows that 
each vertex in $V_1(\Gamma,g)$ has three neighbours in $\mathcal{N}$. Similarly, each vertex in $V_2(\Gamma,g)$ has two neighbours in $\mathcal{N}$. As $g(\Gamma)>4$, we have $g(\Delta)>4$ and hence $3|V_1(\Gamma,g)|+2|V_2(\Gamma,g)|\le |\mathcal{N}|$, because $\Delta(v)\cap \Delta(v')=\emptyset$ for any two distinct vertices $v,v'\in V_1(\Gamma,g)\cup V_2(\Gamma,g)$.
\end{proof}

 An \textbf{amalgam} is a triplet $(L,B,R)$ of groups such that $B=L\cap R$, and its \textbf{index} is the couple $(|L:B|,|R:B|)$.
The amalgam $(L,B,R)$ is said to be \textbf{faithful} if no subgroup of $B$ is normal in $\langle L,R \rangle$; moreover, $(L,B,R)$ is said to be \textbf{$2$-transitive} if the action of $L$ on the right cosets of $B$ by right multiplication is $2$-transitive.

Observe that, if $\Gamma$ is a finite connected $G$-arc-transitive graph of valency $k$, then for any $v\in V\Gamma$ and $w\in \Gamma(v)$, the triplet
\begin{equation*}
	(G_v,G_{vw},G_{\{v,w\}})
\end{equation*}
is a faithful amalgam of index $(k,2)$. 

Finite faithful $2$-transitive amalgams of index $(4,2)$ have been studied in detail by Poto\v{c}nik in \cite{P.amalgams4.2}. We use this work to deduce some properties on $\mathrm{Fix}(E\Gamma,g)$.

\begin{lem}\label{tetraArc2}
	Let $\Gamma$ be a finite connected $4$-valent graph, let $G$ be an $s$-arc-transitive group of automorphisms of $\Gamma$ with $s\geq 2$ and let $g\in G$ fixing pointwise the $s$-arc $(v_0,\ldots,v_{s-1})$. If  $G$ is not $(s+1)$-arc-transitive and $g$ fixes pointwise $\Gamma(v_0)\cup\Gamma(v_{s-1})$, then $g=1$.
\end{lem}
\begin{proof}
	If $G$ is $s$-arc-regular, then $g=1$ because $g$ fixes an $s$-arc. Using~\cite{P.amalgams4.2}, we see that there are $6$ amalgams such that $G$ is not $s$-arc-regular. For each of these remaining amalgams a case-by-case computation shows that the only automorphism leaving the neighbourhood of each end of a given $s$-arc fixed is the identical map.
\end{proof}

\begin{lem}\label{tetraNiceV2}
Let $\Gamma$ be a finite connected $4$-valent graph of girth $g(\Gamma)\geq 5$, let $G$ be a $2$-arc-transitive group of automorphisms of $\Gamma$ such that $G_v^{[1]}\cup G_w^{[2]}$ is a $3$-group, for any two distinct vertices at distance at most $2$, and let $g\in G\setminus\{1\}$. Then $3|V_4(\Gamma,g)|\le 3|V_1(\Gamma,g)|+|V_2(\Gamma,g)|$.
\end{lem}
\begin{proof}
Assume that the vertices in $V_4(\Gamma,g)$ are at pairwise distance more than $2$. Then any two such vertices share no common neighbour. In particular, $\bigcup_{v\in V_4(\Gamma,g)} \Gamma(v)$ has cardinality $4|V_4(\Gamma,g)|$ and is contained in $V_1(\Gamma,g)\cup V_2(\Gamma,g)$. Therefore, $4|V_4(\Gamma,g)|\leq |V_1(\Gamma,g)|+|V_2(\Gamma,g)|$ and the lemma immediately follows in this case.

Assume that there exist two distinct vertices $v$ and $w$ of $V_4(\Gamma,g)$ having distance at most $2$. In particular, $g\in G_v^{[1]}\cap G_w^{[1]}$ and hence $g$ has order a power of $3$, because $G_v^{[1]}\cap G_w^{[1]}$ is a $3$-group. Observe that $V_2(\Gamma,g)=\emptyset$ because an element of order $3$ in a local group  cannot fix exactly two elements.  Let $s\ge 2$ such that $G$ is $s$-arc-transitive, but not $(s+1)$-arc-transitive.

Suppose  $\Gamma[g]$ is not a forest. Then $\Gamma[g]$ contains an $\ell$-cycle $C$. As $V_2(\Gamma,g)=\emptyset$, the vertices of $C$ are elements of $V_4(\Gamma,g)$. From Lemma~\ref{arcTransGirth}, we have $g(\Gamma[g])\ge g(\Gamma)\ge s+1$ and hence, from $C$, we can extract an $s$-arc  whose ends lie in $V_4(\Gamma,g)$, contradicting Lemma~\ref{tetraArc2}. 

Suppose $\Gamma[g]$ is a forest. Let $c$ be the number of connected components of $\Gamma[g]$. From Euler's formula, we have $|V\Gamma[g]|-|E\Gamma[g]|=c$. Clearly, $|V\Gamma[g]|=|V_1(\Gamma,g)|+|V_4(\Gamma,g)|$. Let $\mathcal{S}:=\{(v,w)\in V\Gamma[g]\times V\Gamma[g]\mid \{v,w\}\in E\Gamma[g]\}$. Then
$$2|E\Gamma[g]|=|\mathcal{S}|=\sum_{v\in V\Gamma[g]}|\Gamma[g](v)|=
\sum_{v\in V_1(\Gamma,g)}|\Gamma[g](v)|+\sum_{v\in V_4(\Gamma,g)}|\Gamma[g](v)|
=
|V_1(\Gamma,g)|+4|V_4(\Gamma,g)|.$$
It follows $2|V_3(\Gamma,g)|=|V_1(\Gamma,g)|-2c<|V_1(\Gamma,g)|$.
\end{proof}

\begin{proof}[Proof of Theorem~$\ref{thrm:main1}$ when $\Gamma$ is $2$-arc-transitive] Let $\Gamma$ be a finite connected $4$-valent $2$-arc-transitive graph admitting a non-identity automorphism $g$ with $\mathrm{fpr}(E\Gamma,g)>1/3$ and let $G:=\mathrm{Aut}(\Gamma)$.

If $g(\Gamma)\le 4$, then the proof follows from Lemma~\ref{TetraSmallGirth} and from the remarks at the beginning of Section~\ref{secTetra2AT}. Therefore, for the rest of the proof we suppose that $g(\Gamma)>4$. Since $4|V\Gamma|=2|E\Gamma|$, we have
	\begin{align}\label{align}
\fpr(E\Gamma,g) &= \frac{|F(\Gamma,g)|+|A(\Gamma,g)|}{|E\Gamma|} = \frac{2|F(\Gamma,g)|+2|A(\Gamma,g)|}{4|V\Gamma|}\\
& \leq  \frac{2|F(\Gamma,g)| +|V_1(\Gamma,g)| +2|V_2(\Gamma,g)| +4|V_4(\Gamma,g)|}{8|F(\Gamma,g)|+ 16|V_1(\Gamma,g)|+ 12|V_2(\Gamma,g)|+ 4|V_4(\Gamma,g)|},\nonumber
	\end{align}
where in the last inequality we have used Lemma~\ref{tetraNiceV}.

We claim that, for any two distinct vertices $v,w\in V\Gamma$  at distance at most $2$  one of the following holds
	\begin{enumerate}[(i)]
		\item\label{a1} $G_v^{[1]}\cap G_w^{[1]}$ is a $3$-group; 
		\item\label{a2} the couple $(\Gamma,G)$ defines the amalgam
		\begin{equation*}
			\left( S_3\times S_4, S_3\times S_3, (S_3\times S_3)\rtimes C_2 \right),
		\end{equation*}
moreover, if $d(v,w)=1$, then $G_v^{[1]}\cap G_w^{[1]}=1$ and, if $d(v,w)=2$, then $G_v^{[1]}\cap G_w^{[1]}$ is isomorphic to $C_2$.
	\end{enumerate}
The claim follows with a case-by-case  computation on the finite faithful $2$-transitive amalgams of index $(4,2)$ classified in~\cite{P.amalgams4.2}. We now divide the proof according to~\ref{a1} and~\ref{a2}.	

Suppose that~\ref{a1} holds. From Lemma~\ref{tetraNiceV2}, we have $3|V_4(\Gamma,g)|\le 3|V_1(\Gamma,g)|+|V_2(\Gamma,g)|$. Using this inequality and~\eqref{align}, we obtain $\mathrm{fpr}(E\Gamma,g)\le 1/4<1/3$, which is a contradiction.

Suppose that~\ref{a2} holds. If there exist two distinct vertices $v$ and $w$ in $V_4(\Gamma,g)$ with $d(v,w)=1$, then $g\in G_v^{[1]}\cap G_w^{[1]}=1$, which is a contradiction. Assume there exist two distinct vertices $v$ and $w$ in $V_4(\Gamma,g)$ with $d(v,w)=2$. Then $g\in G_v^{[1]}\cap G_w^{[1]}\cong C_2$ and hence $g$ has order $2$. This implies $V_1(\Gamma,g)=\emptyset$ because an involution in a local group cannot fix only one element. Since the subgraph induced by $\Gamma[g]$ on $V_4(\Gamma,g)$ has no edges and since each vertex in $V_4(\Gamma,g)$ has valency $4$, we deduce $4|V_4(\Gamma,g)|\le |E\Gamma[g]|=|V_2(\Gamma,g)|+2|V_4(\Gamma,g)|$. Using this inequality and~\eqref{align}, we obtain $\mathrm{fpr}(E\Gamma,g)<1/3$, which is a contradiction.
	
Finally, assume that the vertices in $V_4(\Gamma,g)$ are at pairwise distance more than $2$. Then any two such vertices share no common neighbour. In particular, $\bigcup_{v\in V_4(\Gamma,g)} \Gamma(v)$ has cardinality $4|V_4(\Gamma,g)|$ and is contained in $V_1(\Gamma,g)\cup V_2(\Gamma,g)$. Therefore, $4|V_4(\Gamma,g)|\leq |V_1(\Gamma,g)|+|V_2(\Gamma,g)|$.  Using this inequality and~\eqref{align}, we obtain $\mathrm{fpr}(E\Gamma,g)\le 1/4<1/3$, which is a contradiction.
\end{proof}

\subsection{Proof of Theorem~$\ref{thrm:main1}$ when $\Gamma$  is not $2$-arc-transitive}\label{tetranot2}
To conclude the proof of Theorem~\ref{thrm:main1}, we argue by induction on $|V\Gamma|$.

Let $\Gamma$ be a finite connected vertex- and edge-transitive $4$-valent graph admitting a non-identity automorphism $g$ fixing more than $1/3$ of the edges and with $G:=\mathrm{Aut}(\Gamma)$ not $2$-arc-transitive. If $\Gamma$ is isomorphic to a Praeger-Xu graph, then part~\eqref{thrm:main2eq} of Theorem~\ref{thrm:main1} holds. Therefore, for the rest of the argument, we suppose that $\Gamma$ is not isomorphic to $C(r,s)$, for any choice of $r$ and $s$ with $r\ge 3$ and $1\le s\le r-1$.

Let $v\in V\Gamma$. Since $G$ is not $2$-arc-transitive, $G_v^{\Gamma(v)}$ is not $2$-transitive on $\Gamma(v)$. Since $G$ is vertex- and edge-transitive, we obtain that either  $G_v^{\Gamma(v)}$ is transitive or $G_v^{\Gamma(v)}$ has two orbits of cardinality $2$. In both cases, we deduce that $G_v^{\Gamma(v)}$ is a $2$-group. As $\Gamma$ is connected, it follows that $G_v$ is a $2$-group.

If $G$ has no non-identity normal subgroups having cardinality a power of $2$, Theorem~\ref{O2G=1} (applied to the faithful and transitive action of $G$ on $E\Gamma$) contradicts $\mathrm{fpr}(E\Gamma,g)>1/3$. Thus, $G$ has a minimal normal $2$-subgroup $N$. 

As $\Gamma$ is not isomorphic to a Praeger-Xu graph, Lemma~\ref{charPrXu} yields that $N$ acts semi-regularly on $V\Gamma$. Consider the quotient graph $\Gamma/N$ and observe that, as $G$ is vertex- and edge-transitive, $\Gamma/N$ has valency $0$, $1$, $2$ or $4$.

If $\Gamma/N$ has valency 0, then $N$ is transitive on $V\Gamma$. Thus $N$ is vertex-regular on $\Gamma$. As $\Gamma$ is connected of valency $4$, $N$ is generated by at most $4$ elements and hence $|V\Gamma|=|N|$ divides $2^4$.  
If $\Gamma/N$ has valency $1$, then $N$ has two orbits on $V\Gamma$. Moreover,~\cite[Lemma~1.14]{PS.fixedVertex} implies that $|V\Gamma|=2|N|$ divides $128$. In both cases, the statement can be checked computationally by inspecting the candidate graphs from the census of all $4$-valent vertex- and edge-transitive graphs of small order, see \cite{PSV.cubicCensus,PSV.tetraHalfCensus}. 
If $\Gamma/N$ has valency $2$, then we contradict Lemma~\ref{quotientCycle}. Therefore, for the rest of the proof, we may suppose that $\Gamma/N$ has valency $4$.

Observe that $G/N$ acts faithfully as a group of automorphisms on $\Gamma/N$. Moreover, $G/N$ acts vertex- and edge-transitively on $\Gamma/N$, but not $2$-arc-transitively. Observe that $g\notin N$, because the elements in $N$ fix no edge of $\Gamma$. Thus $gN$ is not the identity automorphism of $\Gamma/N$ and, by Lemma~\ref{fprQuotient}, we have $\fpr(E\Gamma/N,Ng)>1/3$. Our inductive hypothesis on $|V\Gamma|$ implies that $\Gamma/N$  is isomorphic to $K_5$ or to  a Praeger-Xu graph $C(r,s)$ with $3s < 2r-3$.

Assume $\Gamma/N\cong K_5$. Now, $\mathrm{Aut}(K_5)=S_5$ and $S_5$ contains a unique conjugacy class of subgroups which are vertex- and edge-transitive, but not $2$-transitive (namely, the Frobenius groups of order $20$). Therefore, $G/N$ is isomorphic to a Frobenius group of order $20$. In particular, as $N$ is an irreducible module for a Frobenius group of order $20$, we get $|N|\le 16$.  We deduce $|V\Gamma|\le 10\cdot 16=160$ and, as above,  the statement can be checked computationally by inspecting the census of all $4$-valent vertex- and edge-transitive graphs of small order.

Assume $\Gamma/N\cong C(r,s)$, for some $r$ and $s$ with $3s<2r-3$. From Lemma~\ref{GconjugateH}, $G/N$ is $\Aut(\Gamma)$-conjugate to a subgroup of $H$ as defined in Section~\ref{secPrXu}. Without loss of generality, we can identify $G/N$ with such subgroup, so that $G/N\leq H$. Now, we first deal with the exceptional case $(r,s)=(4,1)$. As $G/N$ is a $2$-group and $N$ is a minimal normal subgroup of $G$, we deduce $|N|=2$ and hence $|V\Gamma|=|V\Gamma/N||N|=4\cdot 2=8$. Now, the proof follows inspecting the  vertex- and edge-transitive graphs of  order $8$. Therefore, for the rest of the argument, we suppose $(r,s)\ne (4,1)$.  Now, Lemma \ref{gInK} implies $Ng\in K\leq H^+$. Denote by $X$ the group $G/N\cap H^+$. This group is an half-arc-transitive group of automorphisms of $\Gamma/N$ and, since $|H:H^+|=2$, we have $|G/N:X|\leq2$. Denote by $G^+$ the preimage of $X$ with respect to the quotient projection $G\rightarrow G/N$, so that $G^+/N\cong X$. Now, $G^+$ acts half-arc-transitively on $\Gamma$ and, from $Ng\in X$, we see that $g\in G^+$. In particular, replacing $G$ with $G^+$ if necessary, in the rest of our argument we may suppose that  $G=G^+$, that is, $G/N\leq H^+$.

By Lemma~\ref{gInK}, all the edges fixed in $\Gamma/N$ by $gN$ are fixed as arcs. Therefore, all the edges fixed in $\Gamma$ by $g$ are fixed as arcs.

Considering the graph induced by $\Gamma$ on $\mathrm{Fix}(V\Gamma,g)$, we deduce $
		2|\mathrm{Fix}(E\Gamma,g)| \leq 4|\mathrm{Fix}(V\Gamma,g)|$. In particular, if $|\mathrm{Fix}(V\Gamma,g)|\leq |V\Gamma|/3$, then
	\begin{equation*}
\frac{1}{3}<\fpr(E\Gamma, g) = \frac{|\mathrm{Fix}(E\Gamma,g)|}{|E\Gamma|} \leq \frac{2|\mathrm{Fix}(V\Gamma,g)|}{|E\Gamma|} \leq \frac{2|V\Gamma|}{3|E\Gamma|} = \frac{|E\Gamma|}{3|E\Gamma|} = \frac{1}{3},
	\end{equation*}
which is a contradiction. Therefore $\fpr(V\Gamma, g)>1/3$. Now, the hypothesis of Lemma~2.3 in \cite{PS.fixedVertex} are satisfied. Therefore,~\cite[Lemma~2.3]{PS.fixedVertex} implies that $\Gamma$ is a Praeger-Xu graph, which is our final contradiction.

\section{Proof of Theorem~\ref{thrm:main2}}\label{secBuddy}
We now turn our attention to finite connected $3$-valent vertex-transitive graphs. We divide the proof of Theorem~\ref{thrm:main2} in three cases, which we now describe. Let $\Gamma$ be a finite connected $3$-valent vertex-transitive graph, let $G:=\mathrm{Aut}(\Gamma)$ and let $v\in V\Gamma$. The local group $G_v^{\Gamma(v)}$ is a subgroup of the symmetric group of degree $3$ and we divide the proof of Theorem~\ref{thrm:main2} depending on the structure of $G_v^{\Gamma(v)}$. When $G_v^{\Gamma(v)}=1$, the connectivity of $\Gamma$ implies $G_v=1$ and hence $G$ acts regularly on $V\Gamma$. In this case an observation of Sabidussi~\cite{Sabidussi} yields that $\Gamma$ is Cayley graph over $G$. We deal with this case in Section~\ref{secBuddy1}. When $G_v^{\Gamma(v)}$ is cyclic of order $2$,~\cite{PSV.cubicCensus} has established a fundamental relation between $\Gamma$ and a certain finite connected $4$-valent graph;  in Section~\ref{secBuddy2}, we exploit this relation and Theorem~\ref{thrm:main1} to deal with this case. When $G_v^{\Gamma(v)}$ is transitive, $\Gamma$ is arc-transitive and we use the amazing result of Tutte concerning the structure of $G_v$ to deal with this case in Section~\ref{secBuddy3}.

\subsection{Proof of Theorem~$\ref{thrm:main2}$ when the local group is the identity}\label{secBuddy1}
Let $\Gamma$ be a finite connected $3$-valent vertex-transitive graph, let $v\in V\Gamma$, let $G:=\mathrm{Aut}(\Gamma)$  and let $g\in G\setminus\{1\}$. Assume that $G_v^{\Gamma(v)}=1$. Lemma~\ref{Cayley} yields $\mathrm{fpr}(E\Gamma,g)\le 1/3$ and hence Theorem~\ref{thrm:main2} holds in this case.

\subsection{Proof of Theorem~$\ref{thrm:main2}$ when the local group is cyclic of order 2}\label{secBuddy2}
In our proof of this case, we need to refer to two families of $3$-valent Cayley graphs. Given $n\in\mathbb{N}$ with $n\ge 3$, the \textbf{prim} $\mathrm{Pr}_n$ is the Cayley graph
\begin{equation*}
\mathrm{Pr}_n = \mathrm{Cay}\left( \mathbb{Z}_n\times \mathbb{Z}_2, \{(0,1),(1,0),(-1,0)\} \right).
\end{equation*}
Similarly, given $n\in\mathbb{N}$ with $n\ge 2$, the \textbf{M\"{o}bius ladder} $\mathrm{Mb}_n$ is the Cayley graph
\begin{equation*}
\mathrm{Mb}_n = \mathrm{Cay}\left( \mathbb{Z}_{2n}, \{1,n,-1\} \right).
\end{equation*}
For these two classes of graphs the proof of Theorem~\ref{thrm:main2} follows with a computation. When $n\ne 4$, the automorphism group of $\mathrm{Pr}_n$ is isomorphic to $D_n\times C_2$ and, for each $x\in \mathrm{Aut}(\mathrm{Pr}_n)$ with $x\ne 1$, it can be verified that $\mathrm{fpr}(E\mathrm{Pr}_n,x)\le 1/3$, see also Lemma~\ref{Cayley}. The case $n=4$ is exceptional, because $\mathrm{Pr}_4\cong Q_4$ is $2$-arc-transitive and hence $\mathrm{Pr}_4$ is of no concern to us here. Similarly, when $n\notin \{2,3\}$, the automorphism group of $\mathrm{Mb}_n$ is isomorphic to $D_{2n}$ and, for each $x\in \mathrm{Aut}(\mathrm{Mb}_n)$ with $x\ne 1$, it can be verified that $\mathrm{fpr}(E\mathrm{Mb},x)\le 1/3$, again see also Lemma~\ref{Cayley}. The cases $n\in \{2,3\}$ are exceptional, because $\mathrm{Mb}_2\cong K_4$ and $\mathrm{Mb}_3$ are $2$-arc-transitive and hence are of no concern to us here.

Now, let $\Gamma$ be a finite connected $3$-valent vertex-transitive graph not isomorphic to $\mathrm{Pr}_n$ and not isomorphic to $\mathrm{Mb}_n$, let $v\in V\Gamma$, let $G:=\mathrm{Aut}(\Gamma)$ and let $g\in G\setminus\{1\}$ with $\mathrm{fpr}(E\Gamma,g)>1/3$. Assume that $G_v^{\Gamma(v)}$ is cyclic of order $2$.

For a vertex $w \in V\Gamma$, let $w'$ be the neighbour of $w$ such that $\{w\}$ is the orbit of $G_w$ of length $1$. Then
clearly $(w')' = w$ and $G_w = G_{w'}$. Hence, the set $\mathcal{M} := \{\{w, w' \} \mid w \in V\Gamma\}$ is a complete matching of $\Gamma$, while
edges outside $\mathcal{M}$ form a 2-factor $\mathcal{F}$. The group $G$ in its action on $E\Gamma$ fixes setwise both $\mathcal{F}$ and $\mathcal{;}$ and acts transitively on the arcs of each of these two sets. Let $\tilde{\Gamma}$ be the graph with vertex-set $\mathcal{M}$ and two vertices $e_1 , e_2 \in \mathcal{M}$ adjacent if and
only if they are (as edges of $\Gamma$) at distance $1$ in $\Gamma$. The graph $\tilde{\Gamma}$ is then called the \textbf{merge} of $\Gamma$. We may also
think of $\Gamma$ as being obtained by contracting all the edges in $\mathcal{M}$. The group $G$ acts as an arc-transitive
group of automorphisms on $\tilde{\Gamma}$. Moreover, the connected components of the the 2-factor $\mathcal{F}$ gives rise to a
decomposition $\mathcal{C}$ of $E\tilde{\Gamma}$ into cycles.

Since we are assuming that $\Gamma$ is neither a prism nor a M\"{o}bius ladder,~\cite[Lemma~9 and Theorem~10]{PSV.cubicCensus} implies that $\tilde{\Gamma}$ is $4$-valent. Moreover, the action of $G$ on $\tilde{\Gamma}$ is faithful, arc-transitive but not $2$-arc-transitive. Observe also that $\mathrm{fpr}(E\tilde{\Gamma})=\mathrm{fpr}(E\Gamma,g)>1/3$. By Theorem~\ref{thrm:main1}, it follows that $\tilde{\Gamma}\cong C(r,s)$, with $3s<2r-3$ or $\tilde{\Gamma}\cong K_5$. The latter case yields $|V\Gamma|=10$ and the veracity of Theorem~\ref{thrm:main2} follows with an inspection on the finite connected $3$-valent graphs having $10$ vertices. Therefore, suppose $\tilde{\Gamma}\cong C(r,s)$.

In view of~\cite[Theorem~12]{PSV.cubicCensus}, the graph $\Gamma$ can be uniquely reconstructed from $\tilde{\Gamma}$ and the decomposition $\mathcal{C}$ of $E\tilde{\Gamma}$ arising from the $2$-factor $\mathcal{F}$ via the splitting operation defined in Section~\ref{secSplitPrXu}. Therefore, $\Gamma\cong S(C(r,s))$.
\subsection{Proof of Theorem~$\ref{thrm:main2}$ when the local group is transitive}\label{secBuddy3}
Let $\Gamma$ be a finite connected $3$-valent vertex-transitive graph, let $v\in V\Gamma$, let $G:=\mathrm{Aut}(\Gamma)$  and let $g\in G\setminus\{1\}$ with $\mathrm{fpr}(E\Gamma,g)>1/3$. Assume that $G_v^{\Gamma(v)}$ is transitive. In particular, $G$ is $2$-arc-transitive. Let $s\ge 1$ such that $G$ is $s$-arc-transitive and $G$ is not $(s+1)$-arc-transitive. Tutte's theorem~\cite{TutteOnSymCubic} implies  that $G$ is $s$-arc-regular.

Similarly to Definition~\ref{def:1}, we partition $E\Gamma$ with respect to the action of $g$.
\begin{itemize}
\item We let $A(\Gamma,g)$ be the set of edges which are pointwise fixed by $g$, that is, $\{a,b\}\in A(\Gamma,g)$ if and only if $\{a,b\}\in E\Gamma$, $a^g=a$ and  $b^g=b$;
\item we let $F(\Gamma,g):=\mathrm{Fix}(E\Gamma,g)\setminus A(\Gamma,g)$, that is, $\{a,b\}\in F(\Gamma,g)$ if and only if $\{a,b\}\in E\Gamma$, $a^g=b$ and  $b^g=a$;
\item we let $N(\Gamma,g):=E\Gamma\setminus \mathrm{Fix}(E\Gamma,g)$.
\end{itemize}
We let $\Gamma[g]$ denote the subgraph of  $\Gamma$ induced by $\Gamma$ on the vertices which are incident with  edges in $A(\Gamma,g)$. The  edge-set of $\Gamma[g]$ is $A(\Gamma,g)$ and its vertices are  $1$- or $3$-valent. Given $i\in \{1,3\}$, we let $V_i(\Gamma,g)$ denote the set of vertices of $\Gamma[g]$ having valency $i$. 

Suppose $\Gamma[g]$ is not a forest. Then $\Gamma[g]$ contains an $\ell$-cycle $C$. From Lemma~\ref{arcTransGirth}, we have $g(\Gamma[g]) \ge g(\Gamma) \ge s + 1$ and hence, from
$C$, we can extract an $s$-arc $(v_0,v_1,\ldots,v_{s-1})$. As $g$ fixes this $s$-arc and as $G$ is $s$-arc-regular, we deduce $g=1$, which is a contradiction. Therefore $\Gamma[g]$ is a forest. Before proceeding with the proof of Theorem~\ref{thrm:main2}, we prove a preliminary lemma.

\begin{lem}\label{cubicNiceV}
We have $2|F(\Gamma,g)|+3|V_1(\Gamma,g)|+|V_3(\Gamma,g)|\leq |V\Gamma|$.
\end{lem}
\begin{proof}
When $s=1$, the arc-regularity of $G$ implies $V_1(\Gamma,g)=V_3(\Gamma,g)=\emptyset$ and the proof immediately follows. Hence for the rest of the proof we may suppose $s\ge 2$.

We let 
\begin{align*}
\mathcal{F}&:=\{v\in V\Gamma\mid \{v,u\}\in F(\Gamma,g) \textrm{ for some }u\in V\Gamma\},\\
\mathcal{N}&:=\{v\in V\Gamma\mid \{v,u\}\in N(\Gamma,g)\textrm{ for some }u\in V\Gamma\}.
\end{align*}
Since $V_1(\Gamma,g),V_3(\Gamma,g),\mathcal{F},\mathcal{N}$ are pairwise disjoint and since $|\mathcal{F}|=2|F(\Gamma,g)|$, it suffices to show that
$|\mathcal{N}|\ge 2|V_1(\Gamma,g)|$. We divide this proof according to the girth of $\Gamma$.

Suppose $g(\Gamma)\ge 5$. Here, we construct an auxiliary graph $\Delta$. The vertex set of $\Delta$ is $V_1(\Gamma,g)\cup\mathcal{N}$ and we declare a vertex $v\in V_1(\Gamma,g)$ adjacent to a vertex $u\in\mathcal{N}$ if $\{v,u\}\in E\Gamma$. By construction, $\Delta$ is bipartite with parts $V_1(\Gamma,g)$ and $\mathcal{N}$.
Given $v\in V_1(\Gamma,g)$, the automorphism $g$ acts as a $2$-cycle on $\Gamma(v)$. Let $v_1,v_2\in \Gamma(v)$ forming the $2$-cycle of $g$. Then $\{v,v_1\},\{v,v_2\}\in N(\Gamma,g)$ and hence $v_1,v_2\in\mathcal{N}$. This shows that 
each vertex in $V_1(\Gamma,g)$ has two neighbours in $\mathcal{N}$. As $g(\Gamma)\ge 5$, we have $g(\Delta)\ge 5$ and hence $2|V_1(\Gamma,g)|\le |\mathcal{N}|$, because $\Delta(v)\cap \Delta(v')=\emptyset$ for any two distinct vertices $v,v'\in V_1(\Gamma,g)$.

Suppose $g(\Gamma)= 3$. Let $\Gamma(v) = \{w_1 , w_2 , w_3\}$. Without loss of generality, suppose $w_1$ and $w_2$ are adjacent. Since
$G$ is arc-transitive,  $w_i$ is adjacent to $w_j$ for any $i \ne j$. Thus $\Gamma\cong  K_4$. The graph $K_4$ admits no non-identity automorphisms with $\mathrm{fpr}(E\Gamma,g)>1/3$.

Suppose $g(\Gamma)= 4$.	Since $s\ge 2$,~\cite[Theorem~$1.1$ and Table~I]{PeMaKu} implies that $\Gamma$ is isomorphic to either $K_{3,3}$ or $K_{4,4}-4K_2$. In both cases, $\Gamma$ does not admit a non-identity automorphism $g$ with $\mathrm{fpr}(E\Gamma,g)>1/3$. 
\end{proof}

We now resume our proof of Theorem~\ref{thrm:main2}. As $2|E\Gamma|=3|V\Gamma|$, from Lemma~\ref{cubicNiceV}, we have
\begin{align}\label{eq:last}
\fpr(E\Gamma,g)& = \frac{|F(\Gamma,g)|+|A(\Gamma,g)|}{|E\Gamma|} = \frac{2|F(\Gamma,g)|+2|A(\Gamma,g)|}{3|V\Gamma|}\\
& \leq  \frac{2|F(\Gamma,g)| +|V_1(\Gamma,g)| +3|V_3(\Gamma,g)|}{6|F(\Gamma,g)|+ 9|V_1(\Gamma,g)|+ 3|V_3(\Gamma,g)|}.\nonumber
	\end{align}
Let $c$ be the number of connected components of $\Gamma[g]$. From Euler's formula, we have $|V\Gamma[g]|-|E\Gamma[g]|=c$.  Let $\mathcal{S}:=\{(v,w)\in V\Gamma[g]\times V\Gamma[g]\mid \{v,w\}\in E\Gamma[g]\}$. Then
$$2|E\Gamma[g]|=|\mathcal{S}|=\sum_{v\in V\Gamma[g]}|\Gamma[g](v)|=
\sum_{v\in V_1(\Gamma,g)}|\Gamma[g](v)|+\sum_{v\in V_3(\Gamma,g)}|\Gamma[g](v)|
=
|V_1(\Gamma,g)|+3|V_3(\Gamma,g)|.$$
It follows $2|V_3(\Gamma,g)|=|V_1(\Gamma,g)|-2c<|V_1(\Gamma,g)|$. Using~\eqref{eq:last} and this inequality, we obtain $\mathrm{fpr}(E\Gamma,g)\le 1/3$, which is our final contradiction.

\bibliographystyle{alpha}
\bibliography{OnFixedEdgesarXiv}
	
\end{document}